\documentclass[12pt,reqno]{amsart}
\usepackage[english]{babel}
\usepackage[T1]{fontenc}
\usepackage[utf8]{inputenc}
\usepackage[a4paper]{geometry}
\usepackage{graphicx} 

\usepackage{amssymb}
\usepackage{amsmath}
\usepackage{amsthm}
\usepackage{amsfonts}
\usepackage{mathrsfs}
\usepackage{mathtools}
\usepackage{mathdots}
\usepackage{subcaption}
\usepackage{xcolor}
\usepackage[
    colorlinks=true,
    linkcolor=black,
    citecolor=orange,
    urlcolor=teal,
    pdfpagelabels
]{hyperref}
\usepackage{enumitem}
\setlist[enumerate]{itemsep=0.5em, topsep=0.5em}
\setlist[itemize]{itemsep=0.5em, topsep=0.5em}

\newtheorem{theorem}{Theorem}[section]
\newtheorem{lemma}[theorem]{Lemma}
\newtheorem{prop}[theorem]{Proposition}
\newtheorem{corollary}[theorem]{Corollary}
\theoremstyle{definition}
\newtheorem{remark}[]{Remark}

\newcommand{\R}{\mathbb{R}} 
\newcommand{\dd}{\mathop{}\!\mathrm{d}}

\renewcommand{\epsilon}{\varepsilon}
\renewcommand{\theta}{\vartheta}

\newcommand{\abs}[1]{\left\lvert#1\right\rvert} 
\DeclareMathOperator{\sign}{sign}
\newcommand{\F}{\mathcal{F}}
\newcommand{\M}{\mathcal{M}}

\numberwithin{equation}{section}

\mathtoolsset{showonlyrefs}

\title[Unique positive solution for indefinite $p$-Laplacian equations]{A uniqueness result for two-step indefinite $p$-Laplacian equations}
\author{Alberto Cagnetta}
\address{Dipartimento Politecnico di Ingegneria e Architettura, Università degli Studi di Udine,
Via delle Scienze 206, 33100 Udine, Italy}
\email{alberto.cagnetta@uniud.it}

\begin{document}

\begin{abstract}
    Boscaggin, Feltrin and Zanolin conjectured that the Neumann and periodic problems associated with the equation 
    \begin{equation*}
        u''+a(t)u^\gamma=0
    \end{equation*}
    with a two-step indefinite weight $a(t)$ have at most one positive solution for every $\gamma\in\R\setminus\left\{-1,0,1\right\}$. We prove this conjecture and complete the uniqueness picture by also including the logarithmic potential case $\gamma=-1$. More generally, for every $p>1$, we consider the equation
    \begin{equation*}
        \left(\abs{u'}^{p-2}u'\right)'+a(t)u^\gamma=0
    \end{equation*}
    under Neumann or periodic boundary conditions and establish a complete existence and uniqueness classification for all $\gamma\in\R\setminus\left\{0,p-1\right\}$. The necessary mean-value condition $\gamma\int_{0}^{T}a(t)\dd t<0$ is also sufficient for
    existence when $\gamma\leq-1$ or $\gamma>p-1$, while an additional sharp condition on the mean of the weight is required in the intermediate range
    $\gamma\in(-1,p-1)\setminus\{0\}$. Finally, a numerical example provides evidence of multiple positive solutions for the Neumann problem whose weight has the same single change of sign but a non-constant internal profile, highlighting the role of the two-step assumption.
\end{abstract}

\maketitle

\section{Introduction}
\label{sec:introduction}

In this paper, we address the existence and uniqueness of positive solutions of the scalar $p$-Laplacian equation
\begin{equation}
    \label{eq:intro-equation}
    \left(\phi_p(u')\right)'+a(t)u^\gamma=0,
    \qquad
    \phi_{p}(s)\coloneqq \abs{s}^{p-2}s,
\end{equation}
where $p>1$, $\gamma\in\R$, and the coefficient $a$ changes sign. We consider this equation under either Neumann or periodic boundary conditions and focus on the two-step indefinite weight
\begin{equation}
    \label{eq:intro-two-step-weight}
    a(t)=
    \begin{cases}
        a_{+}>0,& t\in[0,\tau),\\
        -a_{-}<0,& t\in[\tau,T],
    \end{cases}
    \qquad 0<\tau<T.
\end{equation}

Following Hess and Kato \cite{HesKat1980}, differential equations with sign-changing weights are commonly called indefinite. The existence and multiplicity of their positive solutions depend on both the nonlinearity and the geometry of the weight.
The $p$-Laplacian is a prototypical nonlinear diffusion operator and arises, for instance, in models of flow through porous media; see \cite[Chapter~3]{KriRadVar2010}. From a stationary reaction-diffusion viewpoint, the sign of $a(t)$ distinguishes spatial regions where the reaction term is positive or negative, while Neumann conditions represent zero flux at the boundary and periodic conditions describe a periodic environment. For $\gamma<0$, the reaction term is singular at $u=0$. Nevertheless, the homogeneity of both the differential operator and the power nonlinearity allows positive and negative exponents to be treated within a common framework.

Inverse-power nonlinearities also occur in related boundary value problems and models. Kaufmann and Medri \cite{KauMed2016} studied indefinite Dirichlet problems for the scalar $p$-Laplacian with a singular power nonlinearity. Inverse-square singularities also arise in stationary electrostatic MEMS models, where the singular term depends on the gap between the membrane and the ground plate \cite{GhoGuo2007}; $p$-Laplacian variants have been considered in \cite{CasEspSci2008}. These connections provide further motivation for including the singular range in the present analysis.

Positive solutions to indefinite boundary value problems have been extensively studied, and a systematic account of topological methods is provided in \cite{Feltrin2018}. In particular, existence results for superlinear Neumann and periodic problems have been obtained through coincidence degree \cite{FelZan2015}. The question of uniqueness has received less attention and remains more difficult beyond the concave setting. Indeed, classical uniqueness results for semilinear Neumann problems rely on suitable concavity assumptions \cite{BanPozTes1988,BroHes1990}, whereas superlinear nonlinearities may give rise to multiple positive solutions when the weight has several changes of sign, with multiplicity closely related to the number of positive humps separated by negative intervals (see, for instance, \cite{Bos2011,FelLopSam2026,FelZan2017} and the references therein). In particular, multiplicity occurs for suitable stepwise coefficients under Neumann boundary conditions \cite{Bos2011}. In the superlinear Dirichlet setting, Feltrin and Troestler \cite{FelTro2026} combine uniqueness and non-degeneracy on the individual positivity intervals to obtain exact multiplicity results under suitable assumptions on the weight. These results highlight the role of the weight geometry in the uniqueness question and suggest focusing on simpler weight configurations to identify conditions ensuring uniqueness. The two-step coefficient \eqref{eq:intro-two-step-weight}, with a single change of sign and constant values on each subinterval, provides a natural setting for this purpose: its structure allows a time-map analysis aimed at a complete existence and uniqueness classification.

The starting point of the present work is the paper by Boscaggin, Feltrin and Zanolin \cite{BosFelZan2020}. For a general $\phi$-Laplacian equation with a weight of the form \eqref{eq:intro-two-step-weight}, they reduced the Neumann problem to a system of two time-map equations and established an abstract uniqueness criterion. Applying this criterion to the $p$-Laplacian with the power nonlinearity $u^\gamma$, they proved existence and uniqueness in the exponent ranges
\begin{equation}
    \gamma\leq\frac{1-2p}{p-1}
    \quad\text{or}\quad
    \gamma>p-1,
\end{equation}
under the necessary sign condition
\begin{equation}
    \label{eq:intro-sign-condition}
    \gamma\int_0^T a(t)\dd t<0.
\end{equation}
For the standard Laplacian, corresponding to $p=2$, they also obtained partial results in the remaining intervals for $\gamma$ and presented numerical evidence for the monotonicity of the relevant time-map quotient. This led them to conjecture uniqueness for $\gamma\in\R\setminus\{-1,0,1\}$. As highlighted in Remark \ref{remark:degenerate-exponents-normalized-maps}, the exponents $\gamma=0$ and $\gamma=1$ are truly degenerate for $p=2$, while $\gamma=-1$ was excluded by their integral parametrization of time maps. Thus, monotonicity remained open in some exponent range, as well as the possibility of a uniform treatment for arbitrary $p>1$.

The main result of this paper, stated in Theorem \ref{thm:main-result}, settles this conjecture, giving a complete existence and uniqueness classification for every $p>1$ and $\gamma\in\R\setminus\{0,p-1\}$, under both Neumann and periodic boundary conditions. In particular, whenever a positive solution exists, it is unique. In the outer ranges $\gamma\leq -1$ and $\gamma>p-1$, condition \eqref{eq:intro-sign-condition} is also sufficient for existence. By contrast, in the intermediate range $\gamma\in(-1,p-1)\setminus\{0\}$, another sharp condition, involving an explicit threshold $\bar{a}_{\gamma,p}^{\ast}$, must be taken into account (e.g. \eqref{eq:def-critic-mean-a}). The result also covers the exponent $\gamma=-1$, which was not included in the original conjecture. For general $p>1$, the degenerate exponents are $\gamma=0$ and $\gamma=p-1$. At these values, the equation and the boundary conditions are invariant under additive translations that preserve positivity and positive rescalings of solutions, respectively, thereby excluding uniqueness.

The proof is carried out entirely within the time-map framework of \cite{BosFelZan2020}. The choice of a stepwise constant weight makes the equation autonomous on the corresponding time intervals and reduces the boundary value problem to two matching time maps. For $\gamma\notin\{-1,0,p-1\}$, a single normalization, valid for every $p>1$, rewrites the ratio of these two maps in terms of integrals over the fixed interval $[0,1]$; strict monotonicity of the quotient yields the uniqueness result without invoking external uniqueness theorems. At $\gamma=-1$, the multiplicative parameter used for the other powers collapses to its limit value and no longer records the ratio between the endpoint values. We therefore introduce a new \emph{additive} parameter which, as explained in Remark \ref{remark:link-rho-upsilon}, describes the logarithmic deviation of the former multiplicative parameter from the limit value. 

The importance of the two-step assumption is illustrated by the numerical example in Section \ref{sec:numerical-multiplicity}. To examine what may happen beyond this class of weights, we consider a piecewise-polynomial weight with the same single change of sign but a non-constant internal profile. For $p=2$ and $\gamma=8$, shooting computations numerically detect at least three positive Neumann solutions. Thus, even when the number and order of the sign intervals are unchanged, modifying the internal geometry of the weight may restore multiplicity. More generally, the sensitivity of indefinite problems to the detailed shape of the weight is visible even at the level of existence in the singular setting \cite{UreCounter2017}.

The paper is organized as follows. Section \ref{sec:setting-main-result} describes the problem, states the complete existence and uniqueness theorem, and recalls the time-map reduction of \cite{BosFelZan2020} in a self-contained form. Section \ref{sec:proof-main-result} introduces the normalized time maps, proves their strict monotonicity, treating the non-logarithmic and logarithmic potential cases separately. Finally, Section \ref{sec:numerical-multiplicity} presents the numerical evidence of multiplicity beyond two-step weights and discusses the role of the internal geometry of the coefficient.

\section{Time-map framework and main result}
\label{sec:setting-main-result}

\subsection{Problem description and notation}
\label{subsec:problem-setting}

We first introduce the framework within which the existence and uniqueness problem will be studied. In particular, we recall the general notation used both in the formulation of the main result and in the time-map reduction developed below. 

Let $T>0$, $p>1$, and $\gamma\in\R$. We consider the scalar $p$-Laplacian equation
\begin{equation}
    \label{eq:main-ODE-phi-laplacian}
    \left(\phi_{p}(u')\right)'
    +a(t)u^{\gamma}=0,
    \qquad t\in(0,T),
\end{equation}
where
\begin{equation}
    \label{eq:def-p-laplacian}
    \phi_{p}(s)\coloneqq\abs{s}^{p-2}s.
\end{equation}
Since we consider only positive solutions, the power nonlinearity $g_{\gamma}(u)\coloneqq u^{\gamma}$ is well defined for every $\gamma\in\R$.

The weight $a(t)$ is assumed to be indefinite and piecewise constant, with exactly one change of sign, as defined in \eqref{eq:intro-two-step-weight}. We assume without loss of generality that the positive part of the weight comes first. Indeed, the symmetric configuration is recovered through the time reversal $t\mapsto T-t$.

We study equation \eqref{eq:main-ODE-phi-laplacian} under either the Neumann boundary conditions
\begin{equation}
    \label{eq:neumann-boundary-conditions}
    \tag{N}
    u'(0)=u'(T)=0
\end{equation}
or the periodic boundary conditions
\begin{equation}
    \label{eq:periodic-boundary-conditions}
    \tag{P}
    u(0)=u(T),
    \qquad
    u'(0)=u'(T).
\end{equation}
By a positive solution of either boundary value problem we mean a function $u\in C^{1}([0,T])$ such that $\phi_{p}(u')\in W^{1,1}(0,T)$, equation \eqref{eq:main-ODE-phi-laplacian} holds almost everywhere in $(0,T)$, the corresponding boundary conditions are satisfied, and
\begin{equation}
    u(t)>0
    \qquad\text{for every }t\in[0,T].
\end{equation}

We also introduce the two constants
\begin{equation}
    \label{eq:def-mu-nu}
    \mu\coloneqq\frac{a_{+}}{a_{+}+a_{-}},
    \qquad
    \nu\coloneqq\frac{a_{-}}{a_{+}}.
\end{equation}
They are related by the identity $\mu(1+\nu)=1$. The mean of the weight over $[0,T]$ is
\begin{equation}
    \label{eq:mean-weight-nu}
        \bar{a}
        \coloneqq\frac{1}{T}\int_{0}^{T}a(t)\dd t
        =
        \frac{a_{+}\tau-a_{-}(T-\tau)}{T}
        =
        a_{+}\left(1-\frac{\tau}{T}\right)
        \left(
            \frac{\tau}{T-\tau}-\nu
        \right).
\end{equation}
Consequently,
\begin{equation}
    \label{eq:sign-mean-weight}
    \sign\bar{a}
    =
    \sign\left(
        \frac{\tau}{T-\tau}-\nu
    \right).
\end{equation}
For later use, we fix the following primitive of $g_{\gamma}$:
\begin{equation}
    \label{eq:def-G-gamma}
    G_{\gamma}(u)\coloneqq
    \begin{cases}
        \displaystyle\frac{u^{\gamma+1}}{\gamma+1},
        & \gamma\neq-1,\\[2mm]
        \log u,
        & \gamma=-1.
    \end{cases}
\end{equation}
For every $\gamma\in\R$, the function $G_{\gamma}$ is strictly increasing on $(0,+\infty)$, and its range and sign are
\begin{equation}
    \label{eq:range-G-gamma}
    G_{\gamma}\left((0,+\infty)\right)
    =
    \begin{cases}
        (0,+\infty), & \gamma>-1,\\
        (-\infty,0), & \gamma<-1,
    \end{cases},\quad G_{-1}\left((0,+\infty)\right)=\R\;\;\text{for }\gamma=-1.
\end{equation}
\noindent Finally, the exponents
\begin{equation}
    \label{eq:degenerate-exponents}
    \gamma=0
    \qquad\text{and}\qquad
    \gamma=p-1
\end{equation}
are structurally degenerate. Indeed, if $\gamma=0$, equation \eqref{eq:main-ODE-phi-laplacian} is invariant under addition of a constant 
\begin{equation}
    u\longmapsto u+C,
\end{equation}
provided that positivity is preserved. Instead, if $\gamma=p-1$, the equation is invariant under the positive rescaling
\begin{equation}
    u\longmapsto\kappa u,
    \quad \kappa>0.
\end{equation}
Both transformations preserve the Neumann and periodic boundary conditions. Hence these invariances rule out uniqueness whenever a positive solution exists and explain why these two exponents are excluded from the main result.

\subsection{Main result}
\label{subsec:main-result}

For $-1<\gamma<p-1$, with $\gamma\neq0$, we define
\begin{align}
    \label{eq:def-K-gamma-p}
    K_{\gamma,p}(\nu)
    &\coloneqq
    \frac{
        \displaystyle
        \int_{0}^{1}
        \frac{1}{
            (1-s)^{\frac{1}{p}}(1+\nu s)^{\frac{\gamma}{\gamma+1}}
        }
        \dd s
    }{
        \displaystyle
        \int_{0}^{1}
        \frac{1}{
            (1-s)^{\frac{1}{p}+\frac{\gamma}{\gamma+1}}
        }
        \dd s
    }=
    \frac{
        \displaystyle
        \int_{0}^{1}
        \left(\frac{1-s}{1+\nu s}\right)^{\frac{\gamma}{\gamma+1}}
        \frac{\dd s}{(1-s)^{\frac{1}{p}+\frac{\gamma}{\gamma+1}}}
    }{
        \displaystyle
        \int_{0}^{1}\frac{1}{(1-s)^{\frac{1}{p}+\frac{\gamma}{\gamma+1}}
        }
        \dd s
    },\\[2mm]
    \label{eq:def-critic-mean-a}
    \bar{a}_{\gamma,p}^{\ast}&\coloneqq 
    a_{+}\frac{\nu(K_{\gamma,p}(\nu)-1)}{1+\nu K_{\gamma,p}(\nu)}=a_{-}\frac{K_{\gamma,p}(\nu)-1}{1+\nu K_{\gamma,p}(\nu)}
\end{align}
The quantity $K_{\gamma,p}(\nu)$ can be read as the weighted mean of the function $\left(\frac{1-s}{1+\nu s}\right)^{\frac{\gamma}{\gamma+1}}$ on $[0,1]$, with weight $(1-s)^{-\frac{1}{p}-\frac{\gamma}{\gamma+1}}$. For the considered values of $\gamma$, both integrals in \eqref{eq:def-K-gamma-p} are finite and strictly positive. Indeed, the factor $1+\nu s$ is bounded away from zero on $[0,1]$ and $1/p<1$. Moreover, since $\gamma>-1$,
\begin{equation*}
    \frac{1}{p}+\frac{\gamma}{\gamma+1}<1
    \quad\text{if and only if}\quad
    \gamma<p-1.
\end{equation*}

We can now state the main result of the paper.

\begin{theorem}
    \label{thm:main-result}
    Let $p>1$ and let $\gamma\in\R\setminus\{0,p-1\}$. Consider problem \eqref{eq:main-ODE-phi-laplacian}, with weight $a(t)$ as defined in \eqref{eq:intro-two-step-weight}, subject to either the Neumann boundary conditions \eqref{eq:neumann-boundary-conditions} or the periodic boundary conditions \eqref{eq:periodic-boundary-conditions}. Then a positive solution exists if and only if one of the corresponding conditions below is satisfied:
    \begin{enumerate}[label=\textup{(\roman*)}]
        \item \label{item:main-super-p-linear} for $p-1<\gamma$: $\bar{a}<0$ or equivalently $\dfrac{\tau}{T-\tau}<\nu$;
        \item \label{item:main-sub-p-linear} for $0<\gamma<p-1$: $\bar{a}_{\gamma,p}^{\ast}<\bar{a}<0$ or equivalently $\nu K_{\gamma,p}(\nu)<\dfrac{\tau}{T-\tau}<\nu$;
        \item \label{item:main-weakly-singular} for $-1<\gamma<0$: $0<\bar{a}<\bar{a}_{\gamma,p}^{\ast}$ or equivalently $\nu<\dfrac{\tau}{T-\tau}<\nu K_{\gamma,p}(\nu)$;
        \item \label{item:main-singular} for $\gamma\leq -1$: $\bar{a}>0$ or equivalently $\dfrac{\tau}{T-\tau}>\nu$.
    \end{enumerate}
    \noindent Whenever a positive solution exists, it is unique.
\end{theorem}

\begin{remark}
    \label{remark:main-result-mean-weight}
    Inequalities \ref{item:main-super-p-linear} and \ref{item:main-singular} are related to a general necessary condition for the existence of positive solutions. More precisely, let $u$ be a positive solution of \eqref{eq:main-ODE-phi-laplacian} satisfying either \eqref{eq:neumann-boundary-conditions} or \eqref{eq:periodic-boundary-conditions}. Dividing the equation by $u^{\gamma}$ and integrating by parts, as in \cite[Remark~2.1]{BosFelZan2020}, gives
    \begin{equation}
        \begin{split}
            \bar{a}&=
            -\frac{1}{T}\int_{0}^{T}
            \frac{
                \left(\phi_{p}(u'(t))\right)'
            }{
                u^{\gamma}(t)
            }
            \dd t\\
            &=-\frac{1}{T}\left[
                \frac{
                    \phi_{p}(u'(t))
                }{
                    u^{\gamma}(t)
                }
            \right]_{0}^{T}
            -\frac{\gamma}{T}
            \int_{0}^{T}
            \phi_{p}(u'(t))
            u'(t)
            u^{-\gamma-1}(t)
            \dd t.
        \end{split}
    \end{equation}
    The boundary term vanishes under both sets of boundary conditions. Since
    \begin{equation}
        \phi_{p}(u'(t))u'(t)
        =
        \abs{u'(t)}^{p},
    \end{equation}
    we obtain
    \begin{equation}
        \label{eq:necessary-mean-identity}
        \bar{a}
        =
        -\frac{\gamma}{T}
        \int_{0}^{T}
        \abs{u'(t)}^{p}
        u^{-\gamma-1}(t)
        \dd t.
    \end{equation}
    By \eqref{eq:main-ODE-phi-laplacian}, a positive solution cannot be constant. Therefore the integral in \eqref{eq:necessary-mean-identity} is strictly positive and, for $\gamma\neq0$, 
    \begin{equation}
        \sign\bar{a}=-\sign\gamma.
    \end{equation}
    Hence the necessary sign condition is
    \begin{equation}
        \label{eq:necessary-sign-condition}
        \gamma\bar{a}<0.
    \end{equation}

    \noindent In this sense, Theorem~\ref{thm:main-result} shows that \eqref{eq:necessary-sign-condition} is also sufficient when
    \begin{equation}
        \gamma\leq-1
        \qquad\text{or}\qquad
        \gamma>p-1.
    \end{equation}
    In the intermediate range $\gamma\in(-1,p-1)\setminus\{0\}$, the sign condition remains necessary, but an additional bound involving the threshold $\bar{a}_{\gamma,p}^{\ast}$ is required.
\end{remark}

\begin{remark}
    \label{remark:comparison-BFZ}
    The main uniqueness result in \cite[Theorem~1.2]{BosFelZan2020} applies to the exponent ranges
    \begin{equation}
        \gamma\leq\frac{1-2p}{p-1}
        \qquad\text{or}\qquad
        \gamma>p-1.
    \end{equation}
    In those ranges, Boscaggin, Feltrin and Zanolin also proved that the necessary condition $\gamma\bar{a}<0$ is sufficient for the existence of a positive solution.

    For the standard Laplacian, corresponding to $p=2$, some additional cases were discussed in \cite[Section~4.1]{BosFelZan2020}. For $0<\gamma<1$, existence and uniqueness were completely characterized in \cite[Proposition~4.1]{BosFelZan2020}, with uniqueness obtained from a general result for concave nonlinearities. For $-1<\gamma<0$, \cite[Proposition~4.2]{BosFelZan2020} showed that, in our notation, 
    \begin{equation}
        \nu<
        \frac{\tau}{T-\tau}
        <\nu K_{\gamma,2}(\nu)
    \end{equation}
    is a sufficient condition for existence. Its necessity and the uniqueness of the solution remained open. Moreover, for $-3<\gamma<-1$, \cite[Proposition~4.3]{BosFelZan2020} established that $\bar{a}>0$ is necessary and sufficient for existence, while uniqueness remained open. In these last two ranges, the authors presented numerical evidence suggesting the strict monotonicity of the corresponding time-map quotient.

    Theorem~\ref{thm:main-result} completes this picture, giving an existence and uniqueness criterion for every $p>1$ and $\gamma\in\R\setminus\{0,p-1\}$. The previous restriction $\gamma\leq(1-2p)/(p-1)$ is relaxed to $\gamma\leq-1$, including the logarithmic potential case. In the intermediate range $\gamma\in(-1,p-1)\setminus\{0\}$, the theorem provides sharp existence conditions and uniqueness for arbitrary $p>1$. In particular, for $p=2$, it proves \cite[Conjecture~4.1]{BosFelZan2020} and, in case~\ref{item:main-weakly-singular}, strengthens \cite[Proposition~4.2]{BosFelZan2020} by establishing both the necessity of the stated existence condition and uniqueness. All cases are treated within the time-map framework, without invoking uniqueness results for concave nonlinearities.
\end{remark}

\subsection{Planar formulation and time-map characterization}
\label{subsec:planar-time-map}

In this subsection, we recall the part of the time-map framework developed by Boscaggin, Feltrin and Zanolin in \cite{BosFelZan2020} that is needed in the proof of Theorem~\ref{thm:main-result}. We first present the construction for a general $\phi$-Laplacian and a positive continuous nonlinearity and then specialize it to the $p$-Laplacian and the power nonlinearity considered in \eqref{eq:main-ODE-phi-laplacian}. This provides a self-contained reduction of the boundary value problem to a system of time-map equations. The further normalization used in the proof of the main result will be introduced later in Section~\ref{sec:proof-main-result}.

As mentioned in Section \ref{sec:introduction}, we consider the more general equation
\begin{equation}
    \label{eq:abstract-phi-laplacian}
    \left(\phi(u')\right)'
    +a(t)g(u)=0,
\end{equation}
where $a$ is the weight defined in \eqref{eq:intro-two-step-weight}, $\phi\colon\Omega\to\R$ is a homeomorphism from an open interval $\Omega\subseteq\R$, with $0\in\Omega$, onto $\R$, satisfying
\begin{equation}
    \label{eq:assumptions-phi}
    \phi(0)=0,
    \qquad
    \phi(s)s>0
    \quad\text{for every }s\in\Omega\setminus\{0\},
\end{equation}
and $g\colon(0,+\infty)\to(0,+\infty)$ is continuous. Equation \eqref{eq:main-ODE-phi-laplacian} is recovered by taking
\begin{equation}
    \phi=\phi_{p},
    \qquad
    g=g_{\gamma}\,\text{ with }\,
    g_{\gamma}(u)=u^{\gamma}.
\end{equation}

Set $h\coloneqq\phi^{-1}$. Introducing the phase variables
\begin{equation}
    x=u,
    \qquad
    y=\phi(u'),
\end{equation}
equation \eqref{eq:abstract-phi-laplacian} can be written as the planar system
\begin{equation}
    \label{eq:abstract-planar-system}
    \begin{cases}
        x'=h(y),\\
        y'=-a(t)g(x).
    \end{cases}
\end{equation}
Since the weight is constant on each side of $t=\tau$, system \eqref{eq:abstract-planar-system} reduces to the two autonomous systems
\begin{equation}
    \label{eq:two-autonomous-systems}
        (\mathcal{S}_{+})\colon
        \begin{cases}
            x'=h(y),\\
            y'=-a_{+}g(x),
        \end{cases}
        t\in[0,\tau),\qquad
        (\mathcal{S}_{-})\colon
        \begin{cases}
            x'=h(y),\\
            y'=a_{-}g(x),
        \end{cases}
        t\in[\tau,T].
\end{equation}
Let $u$ be a positive solution of \eqref{eq:abstract-phi-laplacian} satisfying the Neumann conditions \eqref{eq:neumann-boundary-conditions}, and let $(x,y)=(u,\phi(u'))$ be the corresponding solution of \eqref{eq:abstract-planar-system}. Since
\begin{equation}
    y(0)=y(T)=0,
\end{equation}
integration of the second equation of $(\mathcal{S}_{+})$ gives
\begin{equation}
    y(t)
    =
    -a_{+}\int_{0}^{t}g(x(s))\dd s
    <0,
    \qquad t\in(0,\tau].
\end{equation}
In the same way, integration backward from $T$ gives $y(t)<0$ for every $t\in[\tau,T)$ and, consequently,
\begin{equation}
    \label{eq:y-negative-whole-interval}
    y(t)<0
    \qquad\text{for every }t\in(0,T).
\end{equation}
Since $h$ preserves the sign, it follows that
\begin{equation}
    x'(t)=h(y(t))<0
    \qquad\text{for every }t\in(0,T).
\end{equation}
Thus every positive Neumann solution is strictly decreasing. Setting
\begin{equation}
    \label{eq:def-alpha-xstar-beta}
    \alpha\coloneqq x(0),
    \qquad
    x_{*}\coloneqq x(\tau),
    \qquad
    \beta\coloneqq x(T),
\end{equation}
we obtain
\begin{equation}
    \label{eq:order-alpha-xstar-beta}
    0<\beta<x_{*}<\alpha.
\end{equation}

Now define
\begin{equation}
    \label{eq:def-H-abstract}
    H(y)\coloneqq\int_{0}^{y}h(s)\dd s
\end{equation}
and let $G$ be a primitive of $g$. Since $G'=g>0$, the function $G$ is strictly increasing on $(0,+\infty)$.

Along the first autonomous system $(\mathcal{S}_{+})$, the corresponding energy is preserved:
\begin{equation}
    \label{eq:first-energy-identity}
    H(y)+a_{+}G(x)=a_{+}G(\alpha),
    \qquad t\in[0,\tau].
\end{equation}
In an analogous way, for the second autonomous system $(\mathcal{S}_{-})$ we get
\begin{equation}
    \label{eq:second-energy-identity}
    H(y)-a_{-}G(x)=-a_{-}G(\beta),
    \qquad t\in[\tau,T].
\end{equation}
Evaluating both identities at $t=\tau$ and solving the system with respect to $G(x_{\ast})$ gives the matching relation
\begin{equation}
    \label{eq:matching-G-xstar}
        G(x_{*})
        =
        \frac{
            a_{+}G(\alpha)+a_{-}G(\beta)
        }{
            a_{+}+a_{-}
        }=
        \mu G(\alpha)+(1-\mu)G(\beta).
\end{equation}

\noindent Let
\begin{equation}
    H_{l}\coloneqq\left.H\right|_{(-\infty,0]}.
\end{equation}
Since $h(y)<0$ for $y<0$, the function $H_{l}$ is strictly decreasing and therefore invertible onto its range. Following \cite{BosFelZan2020}, we introduce
\begin{equation}
    \label{eq:def-Lh-Lg}
    \mathcal{L}_{h}\coloneqq h\circ H_{l}^{-1},
    \qquad
    \mathcal{L}_{g}\coloneqq g\circ G^{-1}.
\end{equation}
The function $\mathcal{L}_{h}$ vanishes at the origin and is strictly negative elsewhere on its domain, whereas $\mathcal{L}_{g}>0$ on its domain.

Combining \eqref{eq:first-energy-identity}, \eqref{eq:second-energy-identity}, and \eqref{eq:y-negative-whole-interval}, we obtain
\begin{equation}
    x'
    =
    \mathcal{L}_{h}\left(
        a_{+}G(\alpha)-a_{+}G(x)
    \right),
    \qquad t\in[0,\tau],
\end{equation}
and
\begin{equation}
    x'
    =
    \mathcal{L}_{h}\left(
        a_{-}G(x)-a_{-}G(\beta)
    \right),
    \qquad t\in[\tau,T].
\end{equation}
Integrating along the two decreasing branches gives
\begin{equation}
    \label{eq:time-integrals-first}
    \int_{x_{*}}^{\alpha}
    \frac{
        \dd x
    }{
        -\mathcal{L}_{h}\left(
            a_{+}G(\alpha)-a_{+}G(x)
        \right)
    }
    =
    \tau
\end{equation}
and
\begin{equation}
    \label{eq:time-integrals-second}
    \int_{\beta}^{x_{*}}
    \frac{
        \dd x
    }{
        -\mathcal{L}_{h}\left(
            a_{-}G(x)-a_{-}G(\beta)
        \right)
    }
    =
    T-\tau.
\end{equation}

We now introduce the variables
\begin{equation}
    \label{eq:def-omega-sigma}
    \omega\coloneqq G(\alpha),
    \qquad
    \sigma\coloneqq G(\beta).
\end{equation}
Since $G$ is strictly increasing and $\alpha>\beta$, one has
\begin{equation}
    \omega>\sigma.
\end{equation}
The admissible set for these variables is therefore
\begin{equation}
    \label{eq:def-domain-O}
    \mathcal{O}
    \coloneqq
    \left\{
        (\omega,\sigma)
        \in
        G\left((0,+\infty)\right)
        \times
        G\left((0,+\infty)\right)
        \colon\,
        \omega>\sigma
    \right\}.
\end{equation}
By \eqref{eq:matching-G-xstar},
\begin{equation}
    G(x_{*})
    =
    \mu\omega+(1-\mu)\sigma.
\end{equation}
Performing the change of variable
\begin{equation}
    \theta=G(x),
    \qquad
    \dd x
    =
    \frac{\dd\theta}{\mathcal{L}_{g}(\theta)},
\end{equation}
in \eqref{eq:time-integrals-first} and
\eqref{eq:time-integrals-second}, we define
\begin{equation}
    \label{eq:time-system-M}
    \begin{split}
        \M_{\textup{I}}(\omega,\sigma)
        &\coloneqq
        \int_{\mu\omega+(1-\mu)\sigma}^{\omega}
        \frac{
            \dd\theta
        }{
            -\mathcal{L}_{h}\left(
                a_{+}\omega-a_{+}\theta
            \right)
            \mathcal{L}_{g}(\theta)
        },\\[2mm]
        \M_{\textup{II}}(\omega,\sigma)
        &\coloneqq
        \int_{\sigma}^{\mu\omega+(1-\mu)\sigma}
        \frac{
            \dd\theta
        }{
            -\mathcal{L}_{h}\left(
                a_{-}\theta-a_{-}\sigma
            \right)
            \mathcal{L}_{g}(\theta)
        }.
    \end{split}
\end{equation}

This construction yields the following time-map characterization that will be used throughout the paper.

\begin{prop}[Time-map characterization, {\cite[Theorem~2.1]{BosFelZan2020}}]
    \label{prop:time-map-characterization-M}
    Positive solutions of the Neumann problem associated with \eqref{eq:abstract-phi-laplacian} are in one-to-one correspondence with the pairs $(\omega,\sigma)\in\mathcal{O}$ satisfying
    \begin{equation}
        \label{eq:time-map-system-M}
        \begin{cases}
            \M_{\textup{I}}(\omega,\sigma)=\tau,\\
            \M_{\textup{II}}(\omega,\sigma)=T-\tau.
        \end{cases}
    \end{equation}
    Consequently, the Neumann problem admits exactly one positive solution if and only if system \eqref{eq:time-map-system-M} admits exactly one solution in
    $\mathcal{O}$.
\end{prop}

Suppose now that $G$ has constant sign on $(0,+\infty)$. Then $\omega\neq0$, and we can introduce the parameter
\begin{equation}
    \label{eq:def-rho-BFZ}
    \rho
    \coloneqq
    \frac{
        a_{-}G(\beta)
    }{
        a_{+}G(\alpha)
    }
    =
    \nu\frac{\sigma}{\omega}.
\end{equation}
It follows that $\sigma=\frac{\rho}{\nu}\omega$ and
\begin{equation}
        G(x_{*})=
        \mu\omega+(1-\mu)\sigma=
        \mu(1+\rho)\omega.
\end{equation}

After the change of variable $\theta=\omega\xi$, the two time maps in \eqref{eq:time-system-M} can be written as
\begin{equation}
    \label{eq:time-system-F}
    \begin{split}
        \F_{\textup{I}}(\omega,\rho)
        &\coloneqq
        \omega
        \int_{\mu(1+\rho)}^{1}
        \frac{
            \dd\xi
        }{
            -\mathcal{L}_{h}\left(
                a_{+}\omega(1-\xi)
            \right)
            \mathcal{L}_{g}(\omega\xi)
        },\\[2mm]
        \F_{\textup{II}}(\omega,\rho)
        &\coloneqq
        \omega
        \int_{\frac{\rho}{\nu}}^{\mu(1+\rho)}
        \frac{
            \dd\xi
        }{
            -\mathcal{L}_{h}\left(
                \omega\left(
                    a_{-}\xi-a_{+}\rho
                \right)
            \right)
            \mathcal{L}_{g}(\omega\xi)
        }.
    \end{split}
\end{equation}
Accordingly, system \eqref{eq:time-map-system-M} is equivalent to
\begin{equation}
    \label{eq:time-map-system-F}
    \begin{cases}
        \F_{\textup{I}}(\omega,\rho)=\tau,\\
        \F_{\textup{II}}(\omega,\rho)=T-\tau.
    \end{cases}
\end{equation}
When $\mathcal{L}_{h}$ and $\mathcal{L}_{g}$ are homogeneous, this parametrization separates the variable $\omega$ from the \emph{ratio} parameter $\rho$ and it is remarkably useful, as the second equation of \eqref{eq:time-map-system-F} can be replaced by 
\begin{equation}
    \frac{\F_{\textup{I}}(\omega,\rho)}{\F_{\textup{II}}(\omega,\rho)}=\frac{\tau}{T-\tau}
\end{equation}
and the ratio on the left-hand side no longer depends on $\omega$.

We now return to equation \eqref{eq:main-ODE-phi-laplacian}.
For the $p$-Laplacian, we have
\begin{equation}
    \label{eq:p-laplacian-H-functions}
    \begin{aligned}
        h(s)
        &=\phi^{-1}_{p}(s)=
        \abs{s}^{\frac{p}{p-1}-2}s,\quad h(0)=0\\
        H(y)&=
        \frac{p-1}{p}
        \abs{y}^{\frac{p}{p-1}},\\
        H_{l}^{-1}(\xi)&=
        -\left(
            \frac{p}{p-1}
        \right)^{\frac{p-1}{p}}
        \xi^{\frac{p-1}{p}},
        \qquad \xi\geq0.
    \end{aligned}
\end{equation}
Consequently,
\begin{equation}
    \label{eq:Lh-p-laplacian}
    \mathcal{L}_{h}(\xi)
    =
    -\left(
        \frac{p}{p-1}
    \right)^{\frac{1}{p}}
    \xi^{\frac{1}{p}},
    \qquad \xi\geq0.
\end{equation}
Assume first that $\gamma\neq-1$ and choose $G=G_{\gamma}$ as in \eqref{eq:def-G-gamma}. By \eqref{eq:range-G-gamma}, $G_{\gamma}$ has constant sign, and
\begin{equation}
    G_{\gamma}^{-1}(\xi)
    =
    \abs{\gamma+1}^{\frac{1}{\gamma+1}}
    \abs{\xi}^{\frac{1}{\gamma+1}},
    \qquad
    \xi\in G_{\gamma}\left((0,+\infty)\right).
\end{equation}
It follows that
\begin{equation}
    \label{eq:Lg-power-nonlinearity}
    \mathcal{L}_{g}(\xi)
    =
    \abs{\gamma+1}^{\frac{\gamma}{\gamma+1}}
    \abs{\xi}^{\frac{\gamma}{\gamma+1}}.
\end{equation}

The admissible set for the pair $(\omega,\rho)$ is
\begin{equation}
    \label{eq:def-domain-F}
    \mathcal{D}_{\gamma}
    \coloneqq
    \begin{cases}
        (0,+\infty)\times(0,\nu),
        & \gamma>-1,\\[1mm]
        (-\infty,0)\times(\nu,+\infty),
        & \gamma<-1.
    \end{cases}
\end{equation}
Indeed, if $\gamma>-1$, then
\begin{equation}
\label{eq:order-omega-sigma-gamma_>-1}
    0<\sigma<\omega,
\end{equation}
and hence $0<\rho<\nu$. If $\gamma<-1$, then
\begin{equation}
\label{eq:order-omega-sigma-gamma_<-1}
    \sigma<\omega<0,
\end{equation}
and consequently $\rho>\nu$.

Substituting \eqref{eq:Lh-p-laplacian} and
\eqref{eq:Lg-power-nonlinearity} into
\eqref{eq:time-system-F}, we obtain
\begin{equation}
    \label{eq:functions-FI-FII-special-case}
    \begin{split}
        \F_{\textup{I}}(\omega,\rho)
        &=
        \left(
            \frac{p-1}{p}
        \right)^{\frac{1}{p}}
        \frac{
            \abs{\gamma+1}^{-\frac{\gamma}{\gamma+1}}
            \omega
            \abs{\omega}^{
                -\frac{\gamma}{\gamma+1}-\frac{1}{p}
            }
        }{
            a_{+}^{\frac{1}{p}}
        }
        \int_{\mu(1+\rho)}^{1}
        \frac{
            1
        }{
            \abs{1-\xi}^{\frac{1}{p}}
            \abs{\xi}^{\frac{\gamma}{\gamma+1}}
        }\dd\xi,\\[3mm]
        \F_{\textup{II}}(\omega,\rho)
        &=
        \left(
            \frac{p-1}{p}
        \right)^{\frac{1}{p}}
        \frac{
            \abs{\gamma+1}^{-\frac{\gamma}{\gamma+1}}
            \omega
            \abs{\omega}^{
                -\frac{\gamma}{\gamma+1}-\frac{1}{p}
            }
        }{
            a_{+}^{\frac{1}{p}}
        }
        \int_{\frac{\rho}{\nu}}^{\mu(1+\rho)}
        \frac{
            1
        }{
            \abs{\nu\xi-\rho}^{\frac{1}{p}}
            \abs{\xi}^{\frac{\gamma}{\gamma+1}}
        }\dd\xi.
    \end{split}
\end{equation}
Although the orientation of the integrals in \eqref{eq:functions-FI-FII-special-case} depends on the sign of $\gamma+1$, both time maps are strictly positive on $\mathcal{D}_{\gamma}$.

\begin{remark}[Periodic problem]
    \label{remark:periodic-problem}
    The preceding characterization was derived for the Neumann problem. As shown in \cite[Section~2.2]{BosFelZan2020}, the periodic problem can be reduced to the same setting when $h$ is odd. This condition is automatically satisfied by the $p$-Laplacian.

    Extend the weight $a(t)$ to the whole real line by $T$-periodicity. Let $(x,y)$ be a positive $T$-periodic solution of \eqref{eq:abstract-planar-system}. On $(0,\tau)$, the function $y$ is strictly decreasing, whereas on $(\tau,T)$ it is strictly increasing. Since $h$ is strictly increasing, the same monotonicity properties hold for $x'=h(y)$. Periodicity then implies that $x$ has exactly one maximum in $(0,\tau)$ and exactly one minimum in $(\tau,T)$.

    The symmetry of the autonomous trajectories, which follows from the fact that $h$ is odd, places these critical points at
    \begin{equation}
        t_{\max}=\frac{\tau}{2},
        \qquad
        t_{\min}=\frac{T+\tau}{2}.
    \end{equation}
    Therefore, the restriction of $x$ to
    \begin{equation}
        \left[
            \frac{\tau}{2},
            \frac{T+\tau}{2}
        \right]
    \end{equation}
    is a positive solution satisfying Neumann conditions at both  endpoints.

    Conversely, let $(x,y)$ be a positive Neumann solution on this interval. Define
    \begin{equation}
        \left(\widehat{x}(t),\widehat{y}(t)\right)
        \coloneqq
        \begin{cases}
            \left(x(t),y(t)\right),
            &
            t\in
            \left[
                \frac{\tau}{2},
                \frac{T+\tau}{2}
            \right],\\[3mm]
            \left(x(\tau-t),-y(\tau-t)\right),
            &
            t\in
            \left[
                \frac{\tau-T}{2},
                \frac{\tau}{2}
            \right].
        \end{cases}
    \end{equation}
    The oddness of $h$ and the symmetry of the periodically extended weight imply that $(\widehat{x},\widehat{y})$ solves \eqref{eq:abstract-planar-system}. Its $T$-periodic extension is therefore a positive periodic solution. This construction is reversible and establishes a one-to-one correspondence between positive periodic solutions and positive Neumann solutions on the above half-period interval.

    On that interval, the positive and negative parts of the weight have lengths $\frac{\tau}{2}$ and $\frac{T-\tau}{2}$, respectively, and
    \begin{equation}
        \frac{
            \frac{\tau}{2}
        }{
            \frac{T-\tau}{2}
        }
        =
        \frac{\tau}{T-\tau}.
    \end{equation}
    Hence the parameter conditions of Theorem \ref{thm:main-result} for existence and uniqueness are the same for the Neumann and periodic problems.
\end{remark}

For $\gamma\neq-1$, expressions \eqref{eq:functions-FI-FII-special-case} will be the starting point of the normalization introduced in Section~\ref{subsec:non-logarithmic-case}. When $\gamma=-1$, the primitive $G_{-1}(u)=\log u$ does not have constant sign, so the multiplicative parameter $\rho$ is no longer suitable. The case $\gamma=-1$ will therefore be treated separately, starting from the time maps $\M_{\textup{I}}$ and $\M_{\textup{II}}$ in \eqref{eq:time-system-M}.

\section{Normalized time maps and proof of the main result}
\label{sec:proof-main-result}

\subsection{The case $\gamma\neq-1$}
\label{subsec:non-logarithmic-case}

We begin with the exponents $\gamma\neq-1$, for which the primitive $G_{\gamma}$ has constant sign and the parametrization \eqref{eq:time-system-F} is available. The key step is to normalize both time maps in \eqref{eq:time-system-F} so that their quotient can be expressed in terms of two integrals over the fixed interval $[0,1]$. The same normalization applies both when $\gamma>-1$ and when $\gamma<-1$, although the corresponding parameter domains are different.

For an admissible pair $(\omega,\rho)\in\mathcal{D}_{\gamma}$, we introduce the new parameter
\begin{equation}
    \label{eq:def-lambda-rho}
    \lambda=\lambda(\rho)
    \coloneqq
    \frac{\nu-\rho}{1+\rho}.
\end{equation}
This change of variable is strictly decreasing, since
\begin{equation}
    \label{eq:derivative-lambda-rho}
    \lambda'(\rho)=
    -\frac{1+\nu}{(1+\rho)^{2}}<0,
\end{equation}
and its inverse is given by
\begin{equation}
    \label{eq:inverse-lambda-rho}
    \rho=\frac{\nu-\lambda}{1+\lambda}.
\end{equation}
In view of \eqref{eq:def-domain-F}, the parameter intervals are
\begin{equation}
    \label{eq:domains-rho-lambda}
    \begin{array}{c|c|c}
        & \rho & \lambda \\
        \hline
        \gamma>-1 & (0,\nu) & (0,\nu) \\[1mm]
        \gamma<-1 & (\nu,+\infty) & (-1,0).
    \end{array}
\end{equation}

\noindent For $\lambda\in(-1,\nu)$, we define
\begin{equation}
    \label{eq:def-normalized-integrals}
    \begin{split}
        \mathcal{I}_{0,\gamma,p}(\lambda)
        &\coloneqq
        \int_{0}^{1}
        \frac{1}{
            (1-s)^{\frac{1}{p}}
            (1+\lambda s)^{\frac{\gamma}{\gamma+1}}
        }\dd s,\\[1mm]
        \mathcal{I}_{1,\gamma,p}(\lambda)
        &\coloneqq
        \int_{0}^{1}
        \frac{1}{
            (1-s)^{\frac{1}{p}}
            \left(1-\frac{\lambda}{\nu}s\right)^{\frac{\gamma}{\gamma+1}}
        }\dd s.
    \end{split}
\end{equation}
For each fixed $\lambda\in(-1,\nu)$, the linear factors in \eqref{eq:def-normalized-integrals} are strictly positive for $s\in(0,1)$ and thus, as $p>1$,  both integrals are finite and strictly positive. We then set
\begin{equation}
    \label{eq:def-Q-gamma-p}
    \mathcal{Q}_{\gamma,p}(\lambda)
    \coloneqq
    \nu
    \frac{
        \mathcal{I}_{0,\gamma,p}(\lambda)
    }{
        \mathcal{I}_{1,\gamma,p}(\lambda)
    }.
\end{equation}

\begin{lemma}[Normalized time-map quotient]
    \label{lem:normalized-time-map-quotient}
    Let $\gamma\neq-1$ and let
    $(\omega,\rho)\in\mathcal{D}_{\gamma}$. Then
    \begin{equation}
        \label{eq:ratio-FI-FII-Q}
        \frac{
            \F_{\textup{I}}(\omega,\rho)
        }{
            \F_{\textup{II}}(\omega,\rho)
        }
        =
        \mathcal{Q}_{\gamma,p}(\lambda(\rho)).
    \end{equation}
    Moreover, for every fixed admissible $\rho$, there exists a constant $C_{\gamma,p}(\rho)>0$ such that
    \begin{equation}
        \label{eq:FI-scale-dependence}
        \F_{\textup{I}}(\omega,\rho)
        =
        C_{\gamma,p}(\rho)
        \abs{\omega}^{1-\frac{1}{p}-\frac{\gamma}{\gamma+1}}.
    \end{equation}
\end{lemma}

\begin{proof}
    Set
    \begin{equation}
        \label{eq:def-eta-R}
        \eta\coloneqq\frac{\rho}{\nu},
        \qquad
        R\coloneqq\mu(1+\rho)
        =\frac{1+\rho}{1+\nu}.
    \end{equation}
    If $\gamma>-1$, by \eqref{eq:order-omega-sigma-gamma_>-1}, then $\rho<\nu$, hence 
    \begin{equation}
    \label{eq:order-eta-R-gamma>-1}
        0<\eta<R<1,
    \end{equation}
    whereas, if $\gamma<-1$, by \eqref{eq:order-omega-sigma-gamma_<-1}, then
    \begin{equation}
    \label{eq:order-eta-R-gamma<-1}
        1<R<\eta.
    \end{equation}
    From \eqref{eq:functions-FI-FII-special-case}, up to re-orienting both integrals positively, we obtain the expression
    \begin{equation}
        \label{eq:ratio-before-normalization}
        \frac{
            \F_{\textup{I}}(\omega,\rho)
        }{
            \F_{\textup{II}}(\omega,\rho)
        }
        =
        \nu^{\frac{1}{p}}
        \frac{
            \displaystyle
            \int_{\min\{1,R\}}^{\max\{1,R\}}
            \frac{1}{
                \abs{1-\xi}^{\frac{1}{p}}
                \xi^{\frac{\gamma}{\gamma+1}}
            }
            \dd\xi
        }{
            \displaystyle
            \int_{\min\{\eta,R\}}^{\max\{\eta,R\}}
            \frac{1}{
                \abs{\hat{\xi}-\eta}^{\frac{1}{p}}
                \hat{\xi}^{\frac{\gamma}{\gamma+1}}
            }
            \dd\hat{\xi}
        }.
    \end{equation}
    The identities
    \begin{equation}
        \label{eq:normalized-length-identities}
        \frac{1}{R}-1=\lambda,
        \qquad
        1-\frac{\eta}{R}=\frac{\lambda}{\nu}
    \end{equation}
    hold in both cases, with both sides positive when $\gamma>-1$ and negative when $\gamma<-1$. Furthermore,
    \begin{equation}
        \label{eq:length-ratio-nu}
        \frac{\abs{1-R}}{\abs{R-\eta}}=\nu.
    \end{equation}
    We use the following changes of variable in the numerator and denominator of \eqref{eq:ratio-before-normalization}:
    \begin{equation}
        \label{eq:unified-affine-substitutions}
        \xi=R+(1-R)s,
        \qquad
        \hat{\xi}=R+(\eta-R)s,
        \qquad s\in[0,1].
    \end{equation}
    We get
    \begin{equation}
        \begin{aligned}
            \abs{1-\xi}
            &=\abs{1-R}(1-s),
            &\qquad
            \xi&=R(1+\lambda s),\\
            \abs{\hat{\xi}-\eta}
            &=\abs{R-\eta}(1-s),
            &\qquad
            \hat{\xi}&=R\left(1-\frac{\lambda}{\nu}s\right).
        \end{aligned}
    \end{equation}
    To keep the positive orientation of both integrals, we take the absolute value of the Jacobian of substitutions \eqref{eq:unified-affine-substitutions}, yielding
    \begin{equation}
        \begin{split}
            &\int_{\min\{1,R\}}^{\max\{1,R\}}
            \frac{1}{
                \abs{1-\xi}^{\frac{1}{p}}
                \xi^{\frac{\gamma}{\gamma+1}}
            }
            \dd\xi=
            \abs{1-R}^{1-\frac{1}{p}}
            R^{-\frac{\gamma}{\gamma+1}}
            \mathcal{I}_{0,\gamma,p}(\lambda),\\[2mm]
            &\int_{\min\{\eta,R\}}^{\max\{\eta,R\}}
            \frac{1}{
                \abs{\hat{\xi}-\eta}^{\frac{1}{p}}
                \hat{\xi}^{\frac{\gamma}{\gamma+1}}
            }
            \dd\hat{\xi}=
            \abs{R-\eta}^{1-\frac{1}{p}}
            R^{-\frac{\gamma}{\gamma+1}}
            \mathcal{I}_{1,\gamma,p}(\lambda).
        \end{split}
    \end{equation}
    Taking the ratio between these two expressions and recalling identity \eqref{eq:length-ratio-nu}, we find
    \begin{equation}
        \frac{
            \F_{\textup{I}}(\omega,\rho)
        }{
            \F_{\textup{II}}(\omega,\rho)
        }
        =
        \nu\hspace{1mm}
        \frac{
            \mathcal{I}_{0,\gamma,p}(\lambda)
        }{
            \mathcal{I}_{1,\gamma,p}(\lambda)
        },
    \end{equation}
    which proves \eqref{eq:ratio-FI-FII-Q}.

    The first of the preceding integral identities, together with the formula for $\F_{\textup{I}}$ in
    \eqref{eq:functions-FI-FII-special-case}, also gives
    \begin{equation}
            \F_{\textup{I}}(\omega,\rho)
            =
            \left(
            \left(
                \frac{p-1}{p}
            \right)^{\frac{1}{p}}
            \frac{
                \abs{\gamma+1}^{-\frac{\gamma}{\gamma+1}}
            }{
                a_{+}^{\frac{1}{p}}
            }
            \frac{1-R}{
            \abs{1-R}^{\frac{1}{p}}
            }
            R^{-\frac{\gamma}{\gamma+1}}
            \mathcal{I}_{0,\gamma,p}(\lambda(\rho)) 
            \right)\omega\abs{\omega}^{-\frac{\gamma}{\gamma+1}-\frac{1}{p}}.
    \end{equation}
    By \eqref{eq:def-omega-sigma} and the two inequalities \eqref{eq:order-eta-R-gamma>-1} and \eqref{eq:order-eta-R-gamma<-1}, the two factors $1-R$ and $\omega$ have the same sign (equal to $\sign(\gamma+1)$) and so it is possible to replace their product with the product of their absolute values, getting 
    \begin{equation}
        \begin{split}
            \F_{\textup{I}}(\omega,\rho)&=C_{\gamma,p}(\rho)\abs{\omega}^{1-\frac{1}{p}-\frac{\gamma}{\gamma+1}},
            \\[2mm]
            \text{with }C_{\gamma,p}(\rho)&= \left(
                \frac{p-1}{p}
            \right)^{\frac{1}{p}}
            \frac{
                \abs{\gamma+1}^{-\frac{\gamma}{\gamma+1}}
            }{
                a_{+}^{\frac{1}{p}}
            }
            \abs{1-R}^{1-\frac{1}{p}}
            R^{-\frac{\gamma}{\gamma+1}}
            \mathcal{I}_{0,\gamma,p}(\lambda(\rho)) .
        \end{split}
    \end{equation}
    This proves \eqref{eq:FI-scale-dependence}.
\end{proof}

It follows from \eqref{eq:ratio-FI-FII-Q} that the second equation of the time-map system can be replaced by the scalar equation
\begin{equation}
    \label{eq:scalar-equation-Q}
    \mathcal{Q}_{\gamma,p}(\lambda(\rho))
    =
    \frac{\tau}{T-\tau}.
\end{equation}
Indeed, once $\F_{\textup{I}}(\omega,\rho)=\tau$, equation \eqref{eq:scalar-equation-Q} is equivalent to $\F_{\textup{II}}(\omega,\rho)=T-\tau$.

The advantage of the normalization is that the monotonicity of $\mathcal{Q}_{\gamma,p}$ follows directly by taking its logarithmic derivative, as shown in the following Lemma.

\begin{lemma}[Monotonicity of $\mathcal{Q}_{\gamma,p}$]
    \label{lem:monotonicity-Q-gamma-p}
    Let $\gamma\neq-1,0$. The function $\mathcal{Q}_{\gamma,p}$ is of class $C^{1}$ on $(-1,\nu)$, and
    \begin{equation}
        \label{eq:sign-derivative-Q}
        \sign\mathcal{Q}_{\gamma,p}'(\lambda)
        =
        -\sign\left(
            \frac{\gamma}{\gamma+1}
        \right).
    \end{equation}
    Consequently, the function $\rho\mapsto\mathcal{Q}_{\gamma,p}(\lambda(\rho))$ is strictly increasing if $\gamma<-1$ or $\gamma>0$, and strictly decreasing if $-1<\gamma<0$.
\end{lemma}

\begin{proof}
    On every compact subinterval of $(-1,\nu)$, the two linear factors in \eqref{eq:def-normalized-integrals} are bounded away from zero. The integrands and their partial derivatives with respect to $\lambda$ are therefore dominated by a constant multiple of $(1-s)^{-\frac{1}{p}}$. Since $p>1$, differentiation under the integral sign is justified by the dominated convergence theorem.

    We define the positive weights
    \begin{equation}
        \label{eq:def-weights-w0-w1}
            w_{0}(s,\lambda)
            \coloneqq
            \frac{1}{
                (1-s)^{\frac{1}{p}}
                (1+\lambda s)^{\frac{\gamma}{\gamma+1}}
            },\quad
            w_{1}(s,\lambda)
            \coloneqq
            \frac{1}{
                (1-s)^{\frac{1}{p}}
                \left(1-\frac{\lambda}{\nu}s\right)^{\frac{\gamma}{\gamma+1}}
            }.
    \end{equation}
    The logarithmic derivative (with respect to $\lambda$) of $\mathcal{I}_{0,\gamma,p}(\lambda)$ and $\mathcal{I}_{1,\gamma,p}(\lambda)$ gives
    \begin{equation}
        \label{eq:log-derivative-I0}
        \frac{\dd}{\dd\lambda}
        \log\mathcal{I}_{0,\gamma,p}(\lambda)
        =
        -\frac{\gamma}{\gamma+1}
        \frac{
            \displaystyle
            \int_{0}^{1}
            w_{0}(s,\lambda)
            \frac{s}{1+\lambda s}
            \dd s
        }{
            \displaystyle
            \int_{0}^{1}w_{0}(s,\lambda)\dd s
        }
    \end{equation}
    and
    \begin{equation}
        \label{eq:log-derivative-I1}
        \frac{\dd}{\dd\lambda}
        \log\mathcal{I}_{1,\gamma,p}(\lambda)
        =
        \frac{\gamma}{\gamma+1}
        \frac{
            \displaystyle
            \int_{0}^{1}
            w_{1}(s,\lambda)
            \frac{\frac{s}{\nu}}{1-\frac{\lambda}{\nu}s}
            \dd s
        }{
            \displaystyle
            \int_{0}^{1}w_{1}(s,\lambda)\dd s
        }.
    \end{equation}
    Hence
    \begin{equation}
        \begin{split}
            \frac{\dd}{\dd\lambda}
            \log\mathcal{Q}_{\gamma,p}(\lambda)
            &=-\frac{\gamma}{\gamma+1}
            \left(
                \frac{
                    \displaystyle
                    \int_{0}^{1}
                    w_{0}(s,\lambda)
                    \frac{s}{1+\lambda s}
                    \dd s
                }{
                    \displaystyle
                    \int_{0}^{1}w_{0}(s,\lambda)\dd s
                }
                +
                \frac{
                    \displaystyle
                    \int_{0}^{1}
                    w_{1}(s,\lambda)
                    \frac{\frac{s}{\nu}}{1-\frac{\lambda}{\nu}s}
                    \dd s
                }{
                    \displaystyle
                    \int_{0}^{1}w_{1}(s,\lambda)\dd s
                }
            \right).
        \end{split}
    \end{equation}
    Both quotients within parentheses are strictly positive. Since
    $\mathcal{Q}_{\gamma,p}>0$ and 
    \begin{equation}
        \frac{\mathcal{Q}_{\gamma,p}'(\lambda)}{\mathcal{Q}_{\gamma,p}(\lambda)}=\frac{\dd}{\dd\lambda}
            \log\mathcal{Q}_{\gamma,p}(\lambda),
    \end{equation}
    this proves \eqref{eq:sign-derivative-Q}. The final assertion follows from \eqref{eq:derivative-lambda-rho}.
\end{proof}

In light of Lemmas \ref{lem:normalized-time-map-quotient} and \ref{lem:monotonicity-Q-gamma-p}, in order to analyse the solutions of equation \eqref{eq:scalar-equation-Q}, it is sufficient to check the value of $\mathcal{Q}_{\gamma,p}(\lambda)$ at the boundary.

At the common endpoint $\lambda=0$ for $\gamma<-1$ and $\gamma>-1$ \eqref{eq:domains-rho-lambda}, the two functions $\mathcal{I}_{0,\gamma,p}(\lambda),\mathcal{I}_{1,\gamma,p}(\lambda)$ are continuous at $\lambda=0$ and have the same value there:
\begin{equation}
    \mathcal{I}_{0,\gamma,p}(0)
    =
    \mathcal{I}_{1,\gamma,p}(0)
    =
    \int_{0}^{1}
    \frac{1}{(1-s)^{\frac{1}{p}}}
    \dd s.
\end{equation}
Therefore,
\begin{equation}
    \label{eq:Q-at-zero}
    \mathcal{Q}_{\gamma,p}(0)=\nu.
\end{equation}

\noindent Now we analyse the other boundary value, considering separately the three cases $-1<\gamma<p-1,\gamma<-1$ and $\gamma>p-1$ with $\gamma\neq0$.

Suppose first that
\begin{equation}
    -1<\gamma<p-1,
    \qquad \gamma\neq0.
\end{equation}
In this range,
\begin{equation}
    \frac{1}{p}+\frac{\gamma}{\gamma+1}<1.
\end{equation}
For $0<\lambda<\nu$, both integrands in \eqref{eq:def-normalized-integrals} are bounded, uniformly in $\lambda$, by a constant multiple of
\begin{equation}
    (1-s)^{-\frac{1}{p}}
    +
    (1-s)^{-\frac{1}{p}-\frac{\gamma}{\gamma+1}},
\end{equation}
which is integrable on $(0,1)$ in the present exponent range. The dominated convergence theorem therefore gives
\begin{equation}
    \label{eq:Q-at-nu-K}
        \lim_{\lambda\to\nu^{-}}
        \mathcal{Q}_{\gamma,p}(\lambda)
        =
        \nu
        \frac{
            \displaystyle
            \int_{0}^{1}
            \frac{1}{
                (1-s)^{\frac{1}{p}}
                (1+\nu s)^{\frac{\gamma}{\gamma+1}}
            }
            \dd s
        }{
            \displaystyle
            \int_{0}^{1}
            \frac{1}{
                (1-s)^{\frac{1}{p}+\frac{\gamma}{\gamma+1}}
            }
            \dd s
        }
        =\nu K_{\gamma,p}(\nu).
\end{equation}
Combining \eqref{eq:Q-at-zero}, \eqref{eq:Q-at-nu-K}, and Lemma~\ref{lem:monotonicity-Q-gamma-p}, we also obtain
\begin{equation}
    \label{eq:ordering-K-nu}
    K_{\gamma,p}(\nu)<1
    \quad\text{if }0<\gamma<p-1,
    \qquad
    K_{\gamma,p}(\nu)>1
    \quad\text{if }-1<\gamma<0.
\end{equation}

Next, suppose that $\gamma>p-1$, then
\begin{equation}
    \frac{1}{p}+\frac{\gamma}{\gamma+1}>1.
\end{equation}
As $\lambda\to\nu^{-}$, the numerator $\mathcal{I}_{0,\gamma,p}(\lambda)$ converges to a finite positive value, whereas weights $w_{1}(s,\lambda)$ \eqref{eq:def-weights-w0-w1} form an increasing pointwise family of positive functions. The monotone convergence theorem gives
\begin{equation}
    \mathcal{I}_{1,\gamma,p}(\lambda)
    \longrightarrow
    \int_{0}^{1}
    \frac{1}{
        (1-s)^{\frac{1}{p}+\frac{\gamma}{\gamma+1}}
    }
    \dd s
    =+\infty.
\end{equation}
Consequently,
\begin{equation}
    \label{eq:Q-super-p-endpoint}
    \lim_{\lambda\to\nu^{-}}
    \mathcal{Q}_{\gamma,p}(\lambda)=0.
\end{equation}

Finally, let $\gamma<-1$. In this case, $\frac{1}{p}+\frac{\gamma}{\gamma+1}>1$. As before, for $\lambda\to-1^{+}$, weights $w_{0}(s,\lambda)$ \eqref{eq:def-weights-w0-w1} form an increasing pointwise family of positive functions. The monotone convergence theorem yields again
\begin{equation}
    \mathcal{I}_{0,\gamma,p}(\lambda)
    \longrightarrow
    \int_{0}^{1}
    \frac{1}{
        (1-s)^{\frac{1}{p}+\frac{\gamma}{\gamma+1}}
    }
    \dd s
    =+\infty.
\end{equation}
On the other hand,
\begin{equation}
    \lim_{\lambda\to-1^{+}}
    \mathcal{I}_{1,\gamma,p}(\lambda)
    =
    \int_{0}^{1}
    \frac{1}{
        (1-s)^{\frac{1}{p}}
        \left(1+\frac{s}{\nu}\right)^{\frac{\gamma}{\gamma+1}}
    }
    \dd s <+\infty.
\end{equation}
It follows that
\begin{equation}
    \label{eq:Q-strongly-singular-endpoint}
    \lim_{\lambda\to-1^{+}}
    \mathcal{Q}_{\gamma,p}(\lambda)=+\infty.
\end{equation}

The preceding endpoint analysis, together with Lemma~\ref{lem:monotonicity-Q-gamma-p}, yields the following complete description of the time-map ratio.

\begin{corollary}
    \label{cor:range-normalized-quotient}
    Let $p>1$ and let $\gamma\in\R\setminus\{-1,0,p-1\}$. Then the map
    \begin{equation}
        \rho
        \mapsto
        \mathcal{Q}_{\gamma,p}\left(\lambda(\rho)\right)
    \end{equation}
    has the following properties:
    \begin{enumerate}[label=\textup{(\roman*)}]
        \item if $\gamma>p-1$, it is a strictly increasing bijection
        from $(0,\nu)$ onto $(0,\nu)$;
        \item if $0<\gamma<p-1$, it is a strictly increasing bijection
        from $(0,\nu)$ onto $(\nu K_{\gamma,p}(\nu),\nu)$;
        \item if $-1<\gamma<0$, it is a strictly decreasing bijection
        from $(0,\nu)$ onto $(\nu,\nu K_{\gamma,p}(\nu))$;
        \item if $\gamma<-1$, it is a strictly increasing bijection
        from $(\nu,+\infty)$ onto $(\nu,+\infty)$.
    \end{enumerate}
\end{corollary}

We can now solve the time-map system completely for every nondegenerate exponent $\gamma\neq-1$.

\begin{prop}
    \label{prop:non-logarithmic-exponents}
    For $p>1$ and $\gamma\in\R\setminus\{-1,0,p-1\}$, the conclusions of Theorem \ref{thm:main-result} hold.
\end{prop}

\begin{proof}
    For each of conditions \ref{item:main-super-p-linear}-\ref{item:main-singular}, recalling definition \eqref{eq:def-critic-mean-a} easy computations give the equivalence between the two different formulations. By \eqref{eq:scalar-equation-Q} and Corollary \ref{cor:range-normalized-quotient}, the stated conditions are necessary and sufficient for the existence of an admissible $\rho$ solving the equation \eqref{eq:scalar-equation-Q}. Lemma~\ref{lem:monotonicity-Q-gamma-p} moreover shows that the value of $\rho$ is unique.

    It remains to recover $\omega$. Since $\gamma\neq p-1$, one has
    \begin{equation}
    \label{eq:critical-exponent-no-zero}
        1-\frac{1}{p}-\frac{\gamma}{\gamma+1}=\frac{p-1-\gamma}{p(\gamma+1)}\neq0.
    \end{equation}
    For the uniquely determined $\rho$, formula \eqref{eq:FI-scale-dependence} shows that equation
    \begin{equation}
        \F_{\textup{I}}(\omega,\rho)=\tau
    \end{equation}
    is equivalent to
    \begin{equation}
        \abs{\omega}^{1-\frac{1}{p}-\frac{\gamma}{\gamma+1}}
        =
        \frac{\tau}{C_{\gamma,p}(\rho)}.
    \end{equation}
    By \eqref{eq:critical-exponent-no-zero}, this equation has a unique solution $\omega$, as its sign is given by $\sign(\gamma+1)$. This gives the existence and uniqueness of the positive Neumann solution.

    The periodic case follows from Remark~\ref{remark:periodic-problem}.
\end{proof}

\begin{remark}[Degenerate exponents]
    \label{remark:degenerate-exponents-normalized-maps} 
    The preceding reduction also explains the two non-logarithmic exponents excluded from Theorem~\ref{thm:main-result}. If $\gamma=0$, then
    \begin{equation}
        \mathcal{Q}_{0,p}(\lambda)\equiv\nu.
    \end{equation}
    Consequently, the scalar equation
    \eqref{eq:scalar-equation-Q} can be satisfied only if
    \begin{equation}
        \frac{\tau}{T-\tau}=\nu\quad\text{or equivalently}\quad \bar{a}=0.
    \end{equation}
    In that case, the scalar equation does not determine $\rho$, consistently with the invariance of the original problem under translations that preserve positivity.
    
    Instead, if $\gamma=p-1$, then
    \begin{equation}
        1-\frac{1}{p}-\frac{\gamma}{\gamma+1}=0.
    \end{equation}
    For each fixed admissible $\rho$, both time maps are therefore independent of $\omega$, by \eqref{eq:FI-scale-dependence} and \eqref{eq:ratio-FI-FII-Q}. Whenever the two time-map equations are compatible, the scale of the solution remains undetermined, in agreement with the invariance of the original equation under positive rescaling.
\end{remark}

For $p=2$, the strict monotonicity established above provides an analytic proof of the behaviour of the time-map quotient observed numerically in \cite[Section~4.1]{BosFelZan2020}. The argument applies to every $p>1$ and remains entirely within the time-map framework. The only exponent not covered by this subsection is $\gamma=-1$, which requires an additive normalization and will be treated in the following subsection.

\subsection{The logarithmic case $\gamma=-1$}
\label{subsec:logarithmic-case}

We now consider the exponent $\gamma=-1$. In this case, the primitive of $g_{-1}(u)=u^{-1}$ is 
\begin{equation}
    G_{-1}(u)=\log u,
\end{equation}
which may be seen as a limit of primitives $G_{\gamma}$ for the case $\gamma\neq-1$:
\begin{equation}
    \label{eq:logarithmic-primitive-limit}
    G_{\gamma}(u)-\frac{1}{\gamma+1}
    =
    \frac{u^{\gamma+1}-1}{\gamma+1}
    \longrightarrow
    \log u
    \qquad\text{as }\gamma\to-1.
\end{equation}
The parameter $\rho$ introduced in the preceding subsection is no longer suitable, since $G_{-1}(\alpha)$ may vanish. Indeed, its definition relies on the constant sign and homogeneity of $G_{\gamma}$, whereas $G_{-1}(u)=\log u$ takes values on the whole real line. Moreover, since $\rho$ is defined through the ratio between $\omega$ and $\sigma$, it is invariant under their common multiplication by a positive constant. By contrast, multiplication of $u$ by a positive constant produces a common additive translation of $\log\alpha$ and $\log\beta$. This change from a multiplicative to an additive structure naturally suggests replacing the ratio-based normalization by one based on the difference between $\omega$ and $\sigma$.

Thus, as in \eqref{eq:def-omega-sigma}, we set
\begin{equation}
    \label{eq:logarithmic-primitive-variables}
    \omega\coloneqq\log\alpha,
    \qquad
    \sigma\coloneqq\log\beta,
\end{equation}
and, as $\alpha>\beta$, one has $\omega>\sigma$. Finally, we introduce the new parameter
\begin{equation}
    \label{eq:def-upsilon}
    \upsilon\coloneqq\omega-\sigma
    =
    \log\left(\frac{\alpha}{\beta}\right)>0.
\end{equation}
Moreover,
\begin{equation}
    G_{-1}^{-1}(\theta)=e^{\theta},
    \qquad
    \mathcal{L}_{g}(\theta)=e^{-\theta},
    \qquad \theta\in\R.
\end{equation}
Substituting these identities and \eqref{eq:Lh-p-laplacian} into \eqref{eq:time-system-M}, and then using
\begin{equation}
    \xi=\omega-\theta
    \qquad\text{and}\qquad
    \hat{\xi}=\theta-\sigma
\end{equation}
in the first and second integrals, respectively, gives
\begin{equation}
    \label{eq:logarithmic-time-maps}
    \begin{split}
        \M_{\textup{I}}(\omega,\omega-\upsilon)
        &=
        \left(
            \frac{p-1}{p a_{+}}
        \right)^{\frac{1}{p}}
        e^{\omega}\int_{0}^{(1-\mu)\upsilon}
        \xi^{-\frac{1}{p}}e^{-\xi}\dd\xi,\\[2mm]
        \M_{\textup{II}}(\omega,\omega-\upsilon)
        &=
        \left(
            \frac{p-1}{p a_{-}}
        \right)^{\frac{1}{p}}
        e^{\omega-\upsilon}
        \int_{0}^{\mu\upsilon}
        \hat{\xi}^{-\frac{1}{p}}e^{\hat{\xi}}\dd\hat{\xi}.
    \end{split}
\end{equation}
We note that the variable $\omega$ again cancels out by taking the ratio of the two time maps. After normalizing both intervals of integration to $[0,1]$ we obtain
\begin{equation}
    \label{eq:def-Q-minus-one-p}
    \mathcal{Q}_{-1,p}(\upsilon)
    \coloneqq
    \frac{
        \M_{\textup{I}}(\omega,\omega-\upsilon)
    }{
        \M_{\textup{II}}(\omega,\omega-\upsilon)
    }
    =
    \nu e^{\upsilon}
    \frac{
        \displaystyle
        \int_{0}^{1}
        s^{-\frac{1}{p}}
        e^{-(1-\mu)\upsilon s}\dd s
    }{
        \displaystyle
        \int_{0}^{1}
        s^{-\frac{1}{p}}
        e^{\mu\upsilon s}\dd s
    },
    \quad \upsilon>0,
\end{equation}
where we used \eqref{eq:def-mu-nu}. The next lemma provides the complete qualitative description of this quotient.

\begin{lemma}[Monotonicity and range of $\mathcal{Q}_{-1,p}$]
    \label{lem:logarithmic-time-map-quotient}
    Let $p>1$. The function $\mathcal{Q}_{-1,p}$ is a strictly increasing $C^{1}$ bijection from $(0,+\infty)$ onto $(\nu,+\infty)$.
\end{lemma}

\begin{proof}
    The two integral functions appearing in \eqref{eq:def-Q-minus-one-p} are strictly positive and of class $C^{1}$ on $(0,+\infty)$, with respect to $\upsilon$. Indeed, let $J\subset(0,+\infty)$ be a compact interval. On $J\times(0,1)$, the integrands are dominated by a constant multiple of the integrable function
    \begin{equation}
        s^{-\frac{1}{p}},
    \end{equation}
    whereas their derivatives with respect to $\upsilon$ are dominated by a constant multiple of
    \begin{equation}
        s^{1-\frac{1}{p}}.
    \end{equation}
    Since $p>1$, both functions belong to $L^{1}(0,1)$. Differentiation under the integral sign is therefore justified by the dominated convergence theorem. Since the denominator in \eqref{eq:def-Q-minus-one-p} is strictly positive, it follows that
    \begin{equation}
        \mathcal{Q}_{-1,p}\in C^{1}(0,+\infty).
    \end{equation}

    Define
    \begin{equation}
        I_{0}(\upsilon)
        \coloneqq
        \frac{
            \displaystyle
            \int_{0}^{1}
            s\,s^{-\frac{1}{p}}
            e^{-(1-\mu)\upsilon s}\dd s
        }{
            \displaystyle
            \int_{0}^{1}
            s^{-\frac{1}{p}}
            e^{-(1-\mu)\upsilon s}\dd s
        },
        \qquad
        I_{1}(\upsilon)
        \coloneqq
        \frac{
            \displaystyle
            \int_{0}^{1}
            s\,s^{-\frac{1}{p}}
            e^{\mu\upsilon s}\dd s
        }{
            \displaystyle
            \int_{0}^{1}
            s^{-\frac{1}{p}}
            e^{\mu\upsilon s}\dd s
        }.
    \end{equation}
    Each of these quantities can be regarded as the mean of $s$ with respect to a positive probability density on $(0,1)$. Hence
    \begin{equation}
        \label{eq:inequalities-two-integral-means}
        0<I_{0}(\upsilon)<1,
        \qquad
        0<I_{1}(\upsilon)<1.
    \end{equation}
    As done before, logarithmic differentiation of
    \eqref{eq:def-Q-minus-one-p} yields
    \begin{equation}
        \label{eq:logarithmic-derivative-Q-minus-one-p}
        \frac{\dd}{\dd\upsilon}
        \log\mathcal{Q}_{-1,p}(\upsilon)
        =
        1-(1-\mu)I_{0}(\upsilon)-\mu I_{1}(\upsilon)
        \stackrel{\eqref{eq:inequalities-two-integral-means}}{>}0.
    \end{equation}
    Since $\mathcal{Q}_{-1,p}$ is positive, it follows that
    $\mathcal{Q}_{-1,p}$ is strictly increasing.

    We now determine its behaviour at the endpoints. For $\upsilon\in(0,1]$, the integrands in the numerator and denominator of \eqref{eq:def-Q-minus-one-p} are bounded, respectively, by the integrable functions
    \begin{equation}
        s^{-\frac{1}{p}},
        \qquad
        s^{-\frac{1}{p}}e^{\mu s}.
    \end{equation}
    Therefore, by the dominated convergence theorem,
    \begin{equation}
        \label{eq:Q-minus-one-at-zero}
        \lim_{\upsilon\to0^{+}}
        \mathcal{Q}_{-1,p}(\upsilon)
        =
        \nu.
    \end{equation}

    It remains to determine the behaviour at $+\infty$. For
    $\upsilon\geq1$, one has
    \begin{equation}
        \int_{0}^{1}
        s^{-\frac{1}{p}}e^{-(1-\mu)\upsilon s}\dd s
        \geq
        \int_{0}^{1/\upsilon}s^{-\frac{1}{p}}
        e^{-(1-\mu)\upsilon s}\dd s
        \geq
        e^{-(1-\mu)}\frac{p}{p-1}
        \upsilon^{-\frac{p-1}{p}},
    \end{equation}
    whereas
    \begin{equation}
        \int_{0}^{1}
        s^{-\frac{1}{p}}
        e^{\mu\upsilon s}\dd s
        \leq
        e^{\mu\upsilon}
        \int_{0}^{1}s^{-\frac{1}{p}}\dd s
        =
        e^{\mu\upsilon}\frac{p}{p-1}.
    \end{equation}
    Consequently,
    \begin{equation}
        \begin{aligned}
        \mathcal{Q}_{-1,p}(\upsilon)
        &\geq
        \nu e^{\upsilon}
        \frac{
            e^{-(1-\mu)}
            \upsilon^{-\frac{p-1}{p}}
        }{
            e^{\mu\upsilon}
        }
        =
        \nu
        e^{(1-\mu)(\upsilon-1)}
        \upsilon^{-\frac{p-1}{p}},
        \qquad \upsilon\geq1.
        \end{aligned}
    \end{equation}
    Hence
    \begin{equation}
        \label{eq:Q-minus-one-at-infinity}
        \lim_{\upsilon\to+\infty}
        \mathcal{Q}_{-1,p}(\upsilon)
        =
        +\infty.
    \end{equation}
    The conclusion follows from the continuity, strict monotonicity, and boundary limits of $\mathcal{Q}_{-1,p}$.
\end{proof}

We can now complete the analysis of the $\gamma=-1$ case.

\begin{prop}[The logarithmic exponent]
    \label{prop:logarithmic-exponent}
    Let $p>1$ and $\gamma=-1$. The Neumann and periodic problems associated with \eqref{eq:main-ODE-phi-laplacian} admit a positive solution if and only if
    \begin{equation}
        \label{eq:existence-logarithmic-case}
        \bar{a}>0\quad\text{or equivalently}\quad\frac{\tau}{T-\tau}>\nu.
    \end{equation}
    Whenever a positive solution exists, it is unique.
\end{prop}

\begin{proof}
    In view of \eqref{eq:def-Q-minus-one-p}, the quotient equation for the time-map system is
    \begin{equation}
        \label{eq:quotient-equation-logarithmic-case}
        \mathcal{Q}_{-1,p}(\upsilon)
        =
        \frac{\tau}{T-\tau}.
    \end{equation}
    By Lemma~\ref{lem:logarithmic-time-map-quotient}, this equation has a solution $\upsilon>0$ if and only if \eqref{eq:existence-logarithmic-case} holds, and that solution is unique.

    For the uniquely determined $\upsilon$, the first equation in \eqref{eq:logarithmic-time-maps} gives
    \begin{equation}
        \label{eq:omega-logarithmic-case}
        e^{\omega}
        =
        \tau
        \left(
            \frac{p a_{+}}{p-1}
        \right)^{\frac{1}{p}}
        \left(
            \int_{0}^{(1-\mu)\upsilon}
            \xi^{-\frac{1}{p}}e^{-\xi}\dd\xi
        \right)^{-1}.
    \end{equation}
    Its right-hand side is finite and strictly positive, so it determines a unique $\omega\in\R$. Then it is possible to recover uniquely $\sigma=\omega-\upsilon$, and the time-map characterization in Proposition~\ref{prop:time-map-characterization-M} proves the assertion for the Neumann problem. The periodic problem follows from Remark~\ref{remark:periodic-problem}.
\end{proof}

By \eqref{eq:mean-weight-nu}, condition \eqref{eq:existence-logarithmic-case} is equivalent to $\bar{a}>0$. Since $\gamma=-1$, this is precisely the necessary sign condition $\gamma\bar{a}<0$ in \eqref{eq:necessary-sign-condition}. Thus, also in this case that condition is both necessary and sufficient.

\begin{remark}
    \label{remark:link-rho-upsilon}
    The connection between the parameters $\rho$ and $\upsilon$ used in the two subsections can be made explicit. For $\gamma\neq-1$, the definition \eqref{eq:def-rho-BFZ} and the primitive $G_{\gamma}(u)=u^{\gamma+1}/(\gamma+1)$ give
    \begin{equation}
        \label{eq:rho-upsilon-relation}
        \rho
        =
        \nu\left(\frac{\beta}{\alpha}\right)^{\gamma+1}
        =
        \nu e^{-(\gamma+1)\upsilon},
        \qquad
        \upsilon
        =
        -\frac{1}{\gamma+1}
        \log\left(\frac{\rho}{\nu}\right).
    \end{equation}
    Thus, the additive parameter $\upsilon$ can be regarded as a rescaled logarithmic deviation of $\rho$ from $\nu$.

    For fixed $\alpha>\beta>0$, however,
    \begin{equation}
        \rho\longrightarrow\nu
        \qquad\text{as }\gamma\to-1,
    \end{equation}
    independently of the value of $\alpha/\beta$. Thus the multiplicative parameter becomes constant and loses the information carried by $\upsilon$. This degeneration is consistent with \eqref{eq:logarithmic-primitive-limit}: obtaining $\log u$ from $G_{\gamma}(u)$ requires subtracting the divergent constant $1/(\gamma+1)$, and the homogeneous rescaling property of $G_{\gamma}$ is replaced at $\gamma=-1$ by the additive property of $\log(\cdot)$. Accordingly, relation \eqref{eq:rho-upsilon-relation} cannot be inverted at $\gamma=-1$, whereas the difference $\upsilon=\omega-\sigma$ remains well defined and recovers exactly the missing information.
\end{remark}

\section{Numerical evidence of multiplicity beyond two-step weights}
\label{sec:numerical-multiplicity}

The uniqueness statement in Theorem~\ref{thm:main-result} is obtained under the assumption that the weight is constant on each of its two intervals. It is therefore natural to ask whether the same conclusion can be obtained also for a weight with the same sign pattern, but with a non-constant geometry. The numerical example below suggests that this is not the case. More precisely, we exhibit a piecewise-polynomial weight with exactly one change of sign for which at least three distinct positive Neumann solutions are numerically detected.

We set $T=1$, $p=2$, and $\gamma=8$, and consider
\begin{equation}
    \label{eq:numerical-neumann-problem}
    u''+\widehat a(t)u^{8}=0,
    \qquad
    u'(0)=u'(1)=0,
\end{equation}
where
\begin{equation}
    \label{eq:def-numerical-weight}
    \widehat a(t)
    \coloneqq
    \begin{cases}
        \displaystyle
        \frac{1}{80}
        +\frac{6}{25}t\left(\frac14-t\right),
        & 0\leq t\leq\dfrac14,\\[3mm]
        \displaystyle
        \frac{1}{80},
        & \dfrac14<t\leq\dfrac{11}{20},\\[3mm]
        \displaystyle
        \frac{1}{80}
        +\frac{15}{4}
        \left[
            1-\left(\frac{40t-25}{3}\right)^{2}
        \right],
        & \dfrac{11}{20}<t\leq\dfrac{7}{10},\\[3mm]
        -375,
        & \dfrac{7}{10}<t\leq1.
    \end{cases}
\end{equation}
The weight $\widehat{a}(t)$ is depicted in Figure \textsc{\ref{fig:weight}}. Thus $\widehat{a}(t)$ is positive on $[0,7/10]$ and negative on $(7/10,1]$. In particular, it has the same sign pattern as the two-step weight in \eqref{eq:intro-two-step-weight}. On its positive interval, however, its graph consists of two parabolic arcs separated by a constant segment, followed by a jump to the negative plateau at $\tau=7/10$. Moreover,
\begin{equation}
    \label{eq:mean-numerical-weight}
    \int_{0}^{1}\widehat a(t)\dd t
    =-\frac{35877}{320}<0,
\end{equation}
so that the necessary sign condition \eqref{eq:necessary-sign-condition} is satisfied. The positive part of the weight contains two humps with very different amplitudes: their maxima are
\begin{equation}
    \widehat{a}\left(\frac18\right)=\frac{13}{800},
    \qquad
    \widehat{a}\left(\frac58\right)=\frac{301}{80}.
\end{equation}

The heuristic behind this choice can be seen by integrating \eqref{eq:numerical-neumann-problem}. The Neumann conditions require the balance
\begin{equation}
    \label{eq:numerical-balance-identity}
    \int_{0}^{1}\widehat a(t)u(t)^{8}\dd t=0.
\end{equation}
On the positive part of the weight the solution is concave and its derivative becomes negative, whereas the negative ending part makes the solution convex and may bring the derivative back to zero at $t=1$. The two separated positive humps are encountered at different values of the decreasing solution. Since the factor $u^{8}$ is strongly sensitive to the initial amplitude, varying $u(0)$ changes their relative contribution to \eqref{eq:numerical-balance-identity} and can destroy the monotonicity of the shooting map. The height of the second hump and the negative terminal plateau were chosen so that this balance is attained. This provides only a heuristic strategy: the presence of several critical points of a weight is not asserted to be sufficient for multiplicity in general.

More precisely, to describe the computation, let $A=u(0)>0$ and write
\begin{equation}
    u(t)=A v(t),
    \qquad
    \lambda=A^{7}.
\end{equation}
Problem \eqref{eq:numerical-neumann-problem} is thereby reduced to the normalized initial-value problem
\begin{equation}
    \label{eq:normalized-numerical-shooting}
    v''+\lambda\widehat{a}(t)v^{8}=0,
    \qquad
    v(0)=1,
    \qquad
    v'(0)=0.
\end{equation}
We then consider the shooting function
\begin{equation}
    \label{eq:def-numerical-shooting-map}
    F(\lambda)\coloneqq v'(1;\lambda),
\end{equation}
and every zero of $F$ for which the corresponding trajectory remains positive leads to a positive solution of the Neumann problem. Numerical computations yield three approximate zeros of $F$:
\begin{equation}
    \label{eq:three-numerical-roots}
    \lambda_{1}\simeq46.8687,
    \qquad
    \lambda_{2}\simeq151.2323,
    \qquad
    \lambda_{3}\simeq377.2199.
\end{equation}
The corresponding numerical values are collected in Table~\ref{tab:three-neumann-solutions}.

\begin{table}[htbp]
    \centering
    \renewcommand{\arraystretch}{1.15}
    \begin{tabular}{c|c|c|c}
        $i$ & $\lambda_i$ & $u_i(0)=\lambda_i^{1/7}$ & $u_i(1)$ \\
        \hline
        $1$ & $46.868703$  & $1.732597$ & $0.562974$ \\
        $2$ & $151.232321$ & $2.048226$ & $0.561611$ \\
        $3$ & $377.219872$ & $2.333914$ & $0.550030$
    \end{tabular}\vspace*{2mm}
    \caption{Numerical data for the three positive Neumann solutions of \eqref{eq:numerical-neumann-problem}.}
    \label{tab:three-neumann-solutions}
\end{table}

\begin{figure}
    \centering
    \begin{subfigure}[b]{0.49\textwidth}
        \centering
        \includegraphics[width=\linewidth]{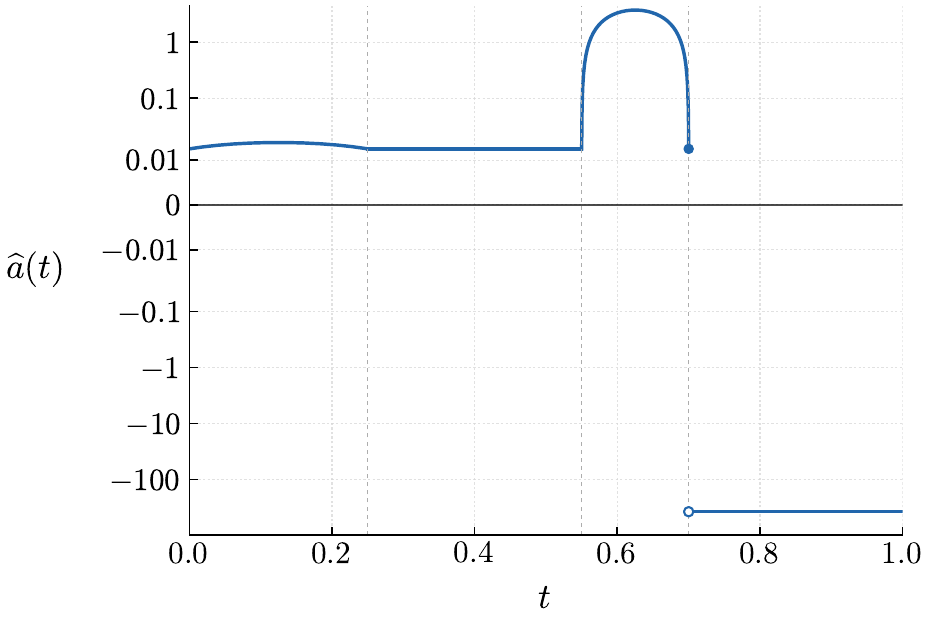}
        \caption{\label{fig:weight}}
    \end{subfigure}
    \begin{subfigure}[b]{0.49\textwidth}
        \centering
        \includegraphics[width=\linewidth]{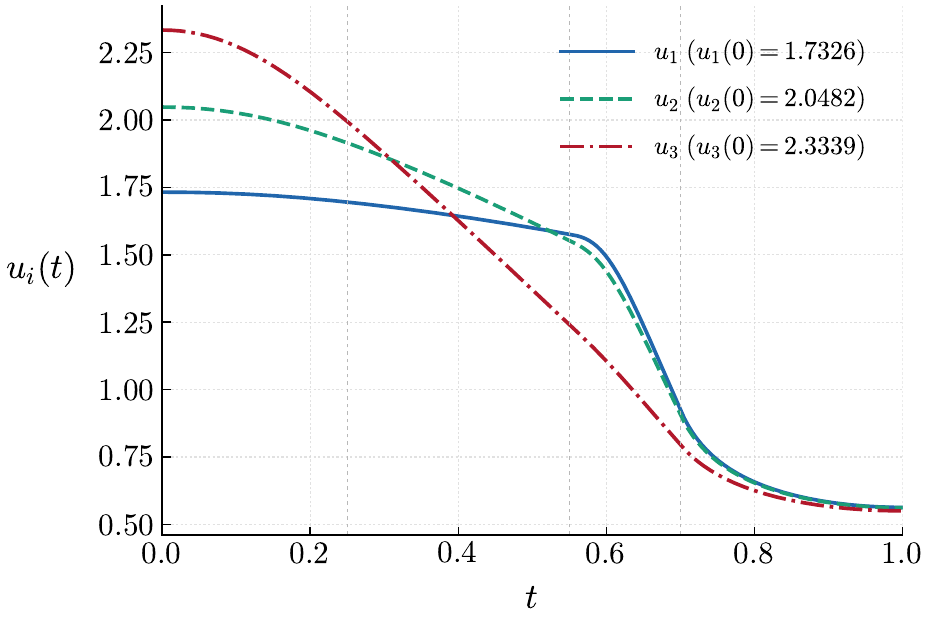}
        \caption{\label{fig:neumann-solutions}}
    \end{subfigure}
    
    \caption{\textsc{\ref{fig:weight}}) The indefinite weight $\widehat{a}(t)$, \textsc{\ref{fig:neumann-solutions}}) Plot of three solutions $u_{i}(t)$ for the Neumann problem \eqref{eq:numerical-neumann-problem} with weight $\widehat{a}(t)$.}
    \label{fig:numerical-neumann-multiplicity}
\end{figure}

These computations provide numerical evidence of at least three distinct positive solutions for equation  \eqref{eq:numerical-neumann-problem} (see Figure \textsc{\ref{fig:neumann-solutions}}), but they do not by themselves exclude further positive solutions. Nevertheless, this example suggests clearly that the uniqueness mechanism of Theorem~\ref{thm:main-result} cannot be expected to depend only on the number and order of the sign intervals and the rigid two-step geometry of the weight plays a pivotal role.\\

\textbf{Acknowledgements.} I would like to thank Guglielmo Feltrin and Paolo Gidoni for helpful and fruitful discussions on this problem and for their suggestions on improving the presentation of this work. A.C. is a member of the Gruppo Nazionale di Fisica Matematica of the Istituto Nazionale di Alta Matematica.

\end{document}